\documentclass{amsart}
\usepackage{cite}
\newcommand{\C}{\ensuremath{\mathbb{C}}}

\theoremstyle{plain}
\newtheorem{theorem}{Theorem}
\newtheorem{proposition}[theorem]{Proposition}
\newtheorem{lemma}[theorem]{Lemma}
\newtheorem{corollary}[theorem]{Corollary}

\theoremstyle{definition}

\newtheorem{remark}[subsection]{Remark}

\newtheorem{nothing*}[subsection]{}

\begin{document}

\markboth{Luka Boc Thaler}{Entire functions with prescribed singular values}

\title{Entire functions with prescribed singular values
 }

\author{Luka Boc Thaler}

\address{L. Boc Thaler: Faculty of Education, University of Ljubljana, SI--1000 Ljubljana, Slovenia. Institute of Mathematics, Physics and Mechanics, Jadranska 19, 1000 Ljubljana, Slovenia.} \email{luka.boc@pef.uni-lj.si}
\thanks{The research program P1-0291 from ARRS, Republic of Slovenia}

\begin{abstract}
 We introduce a new class of entire functions $\mathcal{E}$ which consists of all $F_0\in\mathcal{O}(\C)$ for which there exists a sequence $(F_n)\in \mathcal{O}(\C)$ and a sequence $(\lambda_n)\in\C$ satisfying $F_n(z)=\lambda_{n+1}e^{F_{n+1}(z)}$ for all $n\geq 0$. This new class is closed under the composition and its is dense in the space of all non-vanishing entire functions.  We prove that every closed set $V\subset \C$ containing the origin and at least one more point is the set of singular values of some locally univalent function  in $\mathcal{E}$, hence this new class has non-trivial intersection with both the Speiser class and the Eremenko-Lyubich class of entire functions. As a consequence we provide a new  proof of an old result by Heins which states that  every closed set $V\subset\mathbb{C}$  is the set of singular values of some locally univalent entire function. The novelty of our construction is that these functions are obtained as a uniform limit of a sequence of entire functions, the process under which the set of singular values is not stable. Finally we show that the class $\mathcal{E}$ contains functions with an empty Fatou set and also functions whose Fatou set is non-empty. 
\end{abstract}

\maketitle

%\keywords{entire functions; singular set; asymptotic values; Eremenko-Lyubich class}

%\ccode{Mathematics Subject Classification 2000: 30D05 , 30D15 , 30D20}

\section{introduction}

Let $f:\mathbb{C}\rightarrow\mathbb{C}$ be an entire function. A \emph{critical value} is a point $w=f(z)$ where $z$ is a critical point of $f$, i.e. $f'(z)=0$. A point $w\in \mathbb{C}$ is an \emph{asymptotic value}  if there exists a path $\gamma:[0,\infty) \rightarrow\mathbb{C}$ satisfying $\gamma(t)\rightarrow \infty$ and $f(\gamma(t))\rightarrow w$  as $t\rightarrow\infty$.  By $CV(f)$ and $AV(f)$ we denote respectively the sets of critical values and asymptotic values of $f$. The set of \emph{singular values} is defined as the closed set 
$$S(f):=\overline{AV(f)\cup CV(f)}$$
and recall that  $z\in\mathbb{C}\backslash S(f)$ if and only if there exists a neighbourhood $U$ of $z$ so that $f:f^{-1}(U)\rightarrow U$ is an unbranched covering.

A dynamical system given by the iterates of a function $f$ is to a large extent determined by its singular values (see \cite{sch}). For example we know that every attracting cycle  and every parabolic cycle of Fatou components contains a singular value. In particular this tells us that the function with a finite number of singular values can only have a finite number of attracting/parabolic cycles and this number is bounded above by the number of its  singular values.  It is known that very boundary point of every Siegel disks is a limit point of postsingular points, i.e. forward orbit of singular values. Singular values can also tell us something about the geometry of Fatou component. For example if an entire function has an asymptotic value, then all Fatou components are simply connected. Finally let us mention that in \cite{berg} authors have shown that all limit functions of wandering domains are limit points of the postsingular points. This elementary tool is  useful for proving the absence of wandering domains for some classes of entire functions.

Since the space of entire functions is large and it accommodates a great amount of dynamical variety, it is useful to restrict to smaller classes of functions  in order to obtain strong results.  The class $\mathcal{S}$ (Speiser class) consists of entire functions with finitely many singular values. For example  polynomials and exponential function  belong this class. The Eremenko-Lyubich class $\mathcal{B}$ first defined in \cite{erem}, consists of those entire functions for which the set of singular values is bounded in $\mathbb{C}$.
The above classes are closed under composition which  
is a consequence of the following observation. For  $f=g\circ h$ we have
\begin{equation}\label{eq}
S(f)=g(S(h))\cup S(g)  \text{ and } AV(f)=g(AV(h))\cup AV(g).
\end{equation}
Let us just mention that the class $\mathcal{B}$ exhibits a rich variety of dynamical behaviour and we refer the reader to \cite{six} for a nice survey on the dynamics of functions in this class. 

Even though there are many different techniques that can be used to construct transcendental entire functions, not many of them  give  sufficient control on the set of singular values which is crucial for producing examples in class $\mathcal{B}$. Therefore it is of importance to find new ways to construct entire functions with a control over their set of singular values.
\medskip 

In  \cite{gross}  Gross constructed a locally univalent entire function for which every point in $\mathbb{C}$ is an asymptotic value. In  \cite{heins2} Heins  proved that every Suslin analytic set in $\C$ is the set of asymptotic values of some  locally univalent entire function which in particular implies that every closed subset of $\C$ is the set of singular values of some entire function (see also \cite{heins}). Let us briefly sketch Heins construction.
 
  Given a closed set $A$ we can choose a dense, non-repetitive sequence $(a_n)\in A$ satisfying certain geometric conditions and define $A_n=\{a_1,\ldots,a_n\}$ for all $n\geq 1$. Then we can construct a monotone increasing sequence ($\Omega_n$) of simply-connected Riemann surfaces where each $\Omega_n$ has the conformal radius greater than $n$ and ramification points precisely at $A_n$ ($n=1,2,3,\ldots$). The union 
 $\Omega=\cup_n \Omega_n$ is a simply-connected parabolic Riemann surface, hence by the uniformization theorem  there exist a one-to-one map  $\phi:\C\rightarrow\Omega$. Let $\psi:\Omega\rightarrow \C$ be a locally univalent map (branched covering map).  The entire function $f:=\psi\circ \phi$ is non-constant and it satisfies $(a_n)\subset AV(f)$. With some more effort it is possible to show that actually  $AV(f)=A$. Clearly every function $f$ obtained this way is locally univalent, hence $S(f)=A$.
 
%Using simmilar approach Heins \cite{heins} also proved that given any countable set of complex numbers $A$, there exists an entire locally univalent function $f$ whose set of locally omitted  values (asymptotic values) is equal to $A$, hence $S(f)=\overline{A}$
\medskip 

 Recently Bishop \cite{bishop} introduced a new technique with a good control on the set of singular values, which allows us to construct maps with a rich variety of dynamical behaviour. Given an infinite tree $T$ with a uniformly bounded geometry Bishop's theorem tells us that there exist an entire function $f$ with critical values exactly $\pm 1$ and with no asymptotic values such  that $f^{-1}([-1,1])$ is a quasiconformal perturbation of the tree $T^*$, where $T^*$ is obtained from $T$ by adding some vertices  and  branches.
Entire function $f$ is defined as $f=\psi\circ\phi$ where  $\psi$ quasiregular function satisfying  $\psi^{-1}([-1,1])=T^*$ and  $\phi$ is a quasiconformal mapping given by the measurable Riemann mapping theorem. 

Note that it is not essential that $T$ is a tree as many of the arguments still hold as long as $T$ is a bipartite graph and no two bounded components of  $\C\backslash T$ share a boundary edge. This generalization allowed Bishop to prove that for every bounded, countable sets  $A,B\subset \C$  where $A$ contains at least two points, one can find an entire function $f$ satisfying $CV(f)=A$ and $AV(f)=B$ (\cite[Corollary 9.1.]{bishop}). Note that this method always produces a function which is not locally univalent and  $f\in\mathcal{B}$.

\medskip

Tools like the uniformization theorem and the measurable Riemann mapping theorem come in handy when we wish to construct a function with certain prescribed properties, but once such  function is obtained it is almost impossible to further analyze it.  

Obtaining functions as a uniform limit of a sequence of entire functions $(f_n)$ is much less abstract process then those above, and it enables us to obtain more information about the limit function $f$.   Unfortunately the set of singular values is not stable for small perturbations of the function, hence it is very difficult or even impossible to deduce from the limiting process what the set $S(f)$ would be. For example recall that every entire function is a uniform limit of its Taylor polynomials which in particular implies that the closure\footnote{The space of entire functions $\mathcal{O}(\C)$ is equipped with a  compact-open topology.}
 of the class $\mathcal{S}$ is the space of entire functions $\mathcal{O}(\C)$. 

\medskip

In the present paper we introduce a new class of entire functions $\mathcal{E}$ and we prove that it has several interesting properties (see Section 2). This enables us to give an alternative proof of Heins result, namely we show that for every closed set $V\subset\C$, there exist a locally univalent entire function $f$ satisfying $S(f)=V$. The novelty of our approach with respect to previous constructions is that in our case the function $f$ is obtained as a uniform limit of a sequence $(f_n)$ of non-polynomial entire functions in class $\mathcal{S}$, hence we were able to control $S(f)$ in the limiting process. Functions $f_n$ are given by an explicit formula which enables us to precisely determine each $S(f_n)$ and show that $S(f_n)\subseteq S(f_{n+1})$ for all $n\geq 1$. With some more effort we finally prove that $S(f)=\overline{\cup_n S(f_n)}$. Functions $f$ constructed in this way belong to the new class of entire functions $\mathcal{E}$ which we define next.
\medskip

{\bf Notation:} Let us define $E_{\lambda}(z):=\lambda e^{z}$ and let $E^n_\lambda(z)$ denote the $n$-th iterate of $E_{\lambda}(z)$.  Given a sequence complex numbers $(\lambda_n)$ and integers $0\leq k\leq n$ we further define functions $E_{(k,n)}:=E_{\lambda_{k+1}}\circ\ldots\circ E_{\lambda_n}$ where $E_{(k,k)}(z):=z$.
% Next we define a new class of entire functions.
\medskip

{\bf Definition:} \emph{An entire function $F_0$ belongs to the class $\mathcal{E}$ if and only if there exist a sequence of entire functions $(F_n)$ and a sequence of  complex numbers $(\lambda_n)$ satisfying $F_n(z)=E_{\lambda_{n+1}}({F_{n+1}(z)})$ for all $n\geq 0$.}
\medskip

It should be clear from this definition that if any of the constants $\lambda_n$ above is equal to zero then the function $F_0$ is necessarily constant. Observe that the class $\mathcal{E}$ is closed under the composition and closed under the multiplication with complex numbers. Moreover given any  function $F\in\mathcal{E}$ and $g\in\mathcal{O}(\C)$ and any complex number $\lambda \in\C$  we have $F(g(z))\in\mathcal{E}$ and $E_\lambda(F(z))\in\mathcal{E}$. If  $F\in\mathcal{E}$ and $F\neq  0$ for all $z \in \C$, then $1/F(z)\in\mathcal{E}$.  Clearly $\mathcal{E}$ contains all constant functions and the following theorem which is the main result of this paper will tell us that $\mathcal{E}$ also contains "many" non-trivial functions.

 \begin{theorem}\label{main1} Let $U$ be a closed subset of $\C$ containing $0$ and at least one more point. Let $K\subset \C$ be a compact set and $\epsilon>0$.  There exists a locally univalent entire function $f\in \mathcal{E}$ satisfying $S(f)=U$ and $ \|f(z)-E_1(z)\|_K\leq \epsilon$.
 \end{theorem}

Let $c\in\C$ and define $f_c(z):=f(z)+c$. Observe that $S(f_c)=\{z\in\C\mid z-c\in S(f)\}$ and that $S(e^z)=\{0\}$, hence $S(e^z+c)=\{c\}$. This elementary observation together with Theorem \ref{main1} implies the following corollary.

\begin{corollary}[Heins54] Every closed set in $\C$ is the set of singular values of some locally univalent entire function.
\end{corollary}

\section{Properties of the class $\mathcal{E}$}

Let $\mu,\lambda\in\C$ and observe that $\mu$ is a repelling fixed point of $E_\lambda$ if and only if 
$|\mu|>1$ and $\lambda =\mu e^{-\mu}$. Suppose that this is the case then we know (see \cite[Corollary 8.12]{milnor}) that there exists a non-constant entire function $f$ satisfying 
\begin{equation}\label{poincare1}
f(\mu z)=E_{\lambda}({f(z)}).
\end{equation}
Function $f$ is sometimes called the Poincar\'e function\footnote{In \cite{P1,P2} Poincar\'e has studied the equation $f(\mu z) =R(f(z))$, where $R(z)$ is a rational function and $\mu\in\C$. He proved that, if 0 is a repelling fixed point of $R$ with the multiplier $\mu$, then there exists a meromorphic or entire solution of this equation.}
 of $E_\lambda$ at $\mu$ and it is unique up to the  precomposition by a dilatation, i.e. for any $c\neq0$ the function $f(c\cdot z)$ is also a Poincar\'e function of $E_\lambda$.  From the construction of such function we can deduce that $f$ has no critical values and its set of asymptotic values are precisely the orbit $\{E_\lambda^n(0)\mid n\geq 0\}$. It follows that $S(f)=\overline{\{E_\lambda^n(0)\mid n\geq 0\}}$ therefore $f\in \mathcal{B}$ if and only if  $\{E_\lambda^n(0)\mid n\geq 0\}$ is bounded. Note that the parameters $\lambda$, for which this orbit stays bounded, have been studied by Barker and Rippon  \cite{BR}.  

These Poincar\'e functions $f$ belong to our class $\mathcal{E}$, since by definition $F_k(z)=f(\mu^{-k}z)$ we have 
\begin{equation}\label{ena}F_k=E_\lambda\circ F_{k+1}
\end{equation}
 for all $k\geq0$. As we have seen above, the type of sets that can be realized as the set of singular values of a Poincar\'e function \eqref{poincare1} is very limited. The aim of this paper is to prove that we can replace $\lambda$ in \eqref{ena} by $\lambda_{k+1}\neq 0$ for every $k\geq 0$ and find a solution which satisfies $S(F_0)=\overline{\{E_{(0,n)}(0)\mid n\geq 0\}}$. Let us start with our first result.

\begin{theorem}\label{main2} For every sequence of non-zero complex numbers $(\lambda_n)_{n\geq1}$ there exists a sequence $(F_n)_{n\geq0}$ of locally univalent entire functions satisfying 
\begin{equation}\label{relation}
F_n(z)=E_{\lambda_{n+1}}({F_{n+1}(z)}).
\end{equation} 
for all $n\geq0$. Moreover for every compact set $K$ and $\epsilon>0$ the sequence $(F_n)$ can be chosen so that 
$\|F_0-E_1\|_K<\epsilon$.
\end{theorem}

\begin{proof} Let  $(\lambda_n)_{n\geq1}$ be a sequence of non-zero complex numbers.  For every $k\geq 0$ we define a sequence  of locally univalent entire functions 
\begin{equation}\label{fkn}
f_{k,n}(z):=E_{(k,n)}\left(b_{n}+ \frac{z}{b_0\cdots b_{n-1}}\right),
\end{equation}

where $b_0:=1$ and $b_n:=2\pi m_n i+ \log b_{n-1}-\log{\lambda_{n}}$ with $m_n\in\mathbb{N}$.\footnote{We take a principal branch of logarithm, that is $\log z=\log|z|+i arg(z)$, where $arg(z)\in(-\pi,\pi]$ and $\log 1=0$.}
 Notice that the functions $f_{k,n}$ are only well defined for integers $n\geq k$ and that \begin{equation}\label{eq:f}
f_{k,n}(z)=E_{\lambda_{k+1}}(f_{k+1,n}(z)).
\end{equation}
 We will prove that for sufficiently fast increasing sequence of integers $(m_n)$ the sequence of functions $(f_{k,n})$ converges uniformly on compacts to an entire function $F_k$ as $n\rightarrow \infty$ for all $k\geq 0$. 

Let $(r_n)\nearrow\infty$ be an increasing sequence of integers, and let $\Delta_{r_n}$ denote a closed disk of radius $r_n$ centred at the origin. Observe that for all $n\geq k$ we have
\begin{align*}
\|f_{k,n+1}-f_{k,n}\|_{\Delta_{r_n}} &=\left\|f_{k,n}\left(\left[ E_{\lambda_{n+1}}\left(b_{n+1}+ \frac{z}{b_0\cdots b_{n}}\right)-b_n\right]b_0\cdots b_{n-1}\right)-f_{k,n}(z)\right\|_{\Delta_{r_n}}\\
&=\left\|f_{k,n}\left(\left[E_1\left(\frac{z}{b_0\cdots b_{n}}\right)-1\right]b_0\cdots b_{n}\right)-f_{k,n}(z)\right\|_{\Delta_{r_n}}\\
&\leq M_n\cdot \left\|\left(E_1\left(\frac{z}{b_0\cdots b_{n}}\right)-1\right)b_0\cdots b_{n} -z\right\|_{\Delta_{r_n}}\leq\frac{C_n}{|b_n|} 
\end{align*}

where $M_n$ is a maximal Lipschitz constant on $\Delta_{r_n}$ of the functions $f_{0,n},\ldots, f_{n,n}$. 
Let $(\epsilon_n)$ be a decreasing sequence of real numbers satisfying $0<\epsilon_n\leq 2^{-n}$. Observe that 
\begin{align*}f_{k,n}(z)&=E_{(k,n)}\left(2\pi m_n i+ \log (b_{n-1}\lambda_{n}^{-1})+ \frac{z}{b_0\cdots b_{n-1}}\right)\\
&=E_{(k,n)}\left( \log (b_{n-1}\lambda_{n}^{-1})+ \frac{z}{b_0\cdots b_{n-1}}\right)
\end{align*} therefore the Lipschitz constant on $\Delta_{r_n}$ of the functions $f_{0,n},\ldots, f_{n,n}$ does not depend on $m_n$ hence the same holds for $M_n$. Since $E_1(z)-1-z=O(z^2)$\footnote{By $f(z)=O(g(z))$ we mean that for every compact $K$ there exists a constant $C_K>0$ such that $|f(z)|\leq C_K|g(z)|$ for all $z\in K$.}   and $b_1,\ldots, b_{n-1}$ do not depend on $m_n$, we can make sure that $C_n$ is independent from $m_n$ as well. Therefore by choosing $m_n$ sufficiently large we obtain 
\begin{equation}\label{ineq}
\|f_{k,n+1}-f_{k,n}\|_{\Delta_{r_n}}\leq \epsilon_n,
\end{equation}
for every $n\geq k$. Clearly this implies that $(f_{k,n})$ converges uniformly on compacts to an entire function $F_k$ as $n\rightarrow \infty$.

Recall that the uniform limit of locally univalent functions is either constant or locally univalent. If we show that $F_0$ is non-constant, then since for every $k\geq 0$ we have the equality $F_k(z)=E_{\lambda_{k+1}}(F_{k+1}(z))$ (see \eqref{eq:f}), it follows that functions $F_k$ are also non-constant and locally univalent. 

First observe that $E_{\lambda_k}(b_k)=b_{k-1}$ for every $k\geq0$,  hence  $f_{k,n}(0)=b_k$ and in particular $f_{0,n}(0)=1$. Since 
$f_{0,n}'(z)=\frac{1}{b_0\cdots b_{n-1}}\prod_{k=0}^{n-1}f_{k,n}(z)$ it follows that $f_{0,n}'(0)=1$.
This shows that $f_{0,n}(z)=1+z+O(z^2)$ for all $n>0$ hence the same holds for its limit, i.e. $F_0(z)=1+z+O(z^2)$ and therefore $F_0$ is non-constant locally univalent function.
 
 For the last statement of the theorem observe that $f_{0,1}(z)=E_1(z)$. If we take $r_1>0$ sufficiently large so that $K\subset \Delta_{r_1}$ and if the sequence $(\epsilon_n)$ is chosen so that $\sum_n\epsilon_n<\epsilon$ then it follows that $\|F_0-E_1\|_K<\epsilon$.

\end{proof}

\begin{remark} We have constructed a sequence of locally univalent entire functions satisfying $F_n(z)=E_{\lambda_{n+1}}(F_{n+1}(z))$ for all $n\geq 0$. Since this implies  
 $F_{0}(z)=E_{(0,n)}(F_n(z))$ for all $n\geq 0$ we can deduce from  \eqref{eq}  that
$\{E_{(0,k)}(0)\mid k\geq 0\}\subseteq S(F_{0}).$  In the last section of this paper we will prove that we can choose the sequence integers $m_n$ in the definition of $b_n$'s above so that the function $F_0$  satisfies $S(F_0)=\overline{\{E_{(0,k)}(0)\mid k\geq 0\}}$.  
\end{remark}

Next we prove that for a given non-constant function $F_0\in\mathcal{E}$ the pair of sequences $(F_n)$ and $(\lambda_n)$ given by the definition of $\mathcal{E}$ is unique.

\begin{lemma}\label{unique} Let  $F_0\in\mathcal{E}$ be a non-constant function. Then there exist a unique pair of sequences of non-zero complex numbers $(\lambda_n)$ and entire functions $(F_n)$ satisfying $F_n(z)=E_{\lambda_{n+1}}({F_{n+1}(z)})$ for all $n\geq 0$.
\end{lemma}
\begin{proof} Suppose there are sequences of non-zero complex numbers $(\lambda_n)$, $(\mu_n)$ and entire functions $(F_n)$, $(G_n)$ 
satisfying $$F_n(z)=E_{\lambda_{n+1}}({F_{n+1}(z)}),\qquad G_n(z)=E_{\mu_{n+1}}({G_{n+1}(z)})$$ for all $n\geq 0$ where $G_0= F_0$.

We will prove that for any given $n\geq0$ the equality $G_n= F_n$ implies that $\mu_{n+1}=\lambda_{n+1}$ and $G_{n+1}= F_{n+1}$.

If $G_n= F_n$ then by the definition $E_{\lambda_{n+1}}({F_{n+1}(z)})=E_{\mu_{n+1}}({G_{n+1}(z)})$ and 
$$e^{F_{n+1}(z)-G_{n+1}(z)}=\frac{\mu_{n+1}}{\lambda_{n+1}}$$
hence $F_{n+1}(z)-G_{n+1}(z)\equiv C\in\C$. On the other hand we have
\begin{align*}
F_{n+1}(z)-G_{n+1}(z)&=E_{\lambda_{n+2}}({F_{n+2}(z)})-E_{\mu_{n+2}}({G_{n+2}(z)})\\
&=\lambda_{n+2}e^{F_{n+2}(z)}\left(1-\frac{\mu_{n+2}}{\lambda_{n+2}}e^{G_{n+2}(z)-F_{n+2}(z)}\right).
\end{align*}
Since entire functions $F_j$ and $G_j$ are non-constant for all $n\leq j\leq n+2$ it follows from  
$$\lambda_{n+2}e^{F_{n+2}(z)}\left(1-\frac{\mu_{n+2}}{\lambda_{n+2}}e^{G_{n+2}(z)-F_{n+2}(z)}\right)=C,$$
that $C=0$ and $G_{n+1}= F_{n+1}$, hence $\lambda_{n+1}=\mu_{n+1}$.
\end{proof}

Let $(K_j)_{j\geq 0}$ be a compact exhaustion of  $\C$ and let $0<\epsilon<1$. For non-constant functions $F_0, G_0 \in \mathcal{E}$ we define:

\begin{equation}\label{metric}
d(F_0,G_0):=\sum_{k=0}^{\infty}\frac{\|F_k-G_k\|_{K_k}}{1+\|F_k-G_k\|_{K_k}}\epsilon^{k},
\end{equation}

where $(F_k)$ and $(G_k)$ are as in Lemma \ref{unique}, associated sequences to $F_0$ and $G_0$ respectively. Observe that $d(\cdot,\cdot)$ is a well defined metric on $\mathcal{E}\backslash\C$.

\begin{lemma}\label{complete} Let $(F^j_0)\in\mathcal{E}\backslash \C$ be a Cauchy sequence with respect to the metric $d$. Then $(F^j_0)$
converges uniformly on compacts to an entire function $F_0 \in\mathcal{E}$. \footnote{Here the superscript $F_0^j$ is used to denote the element of the sequence and should not be confused with the iterate of the function.} 
\end{lemma}
\begin{proof}
Let $(F^j_0)\in\mathcal{E}$ be a sequence of non-constant functions and let $(\lambda_k^j)$ and $(F_k^j)$ be  sequences of non-zero complex numbers and entire functions associated to each $F_0^j$ given by Lemma \ref{unique}. Since $(F^j_0)$ is a Cauchy sequence with respect to the metric $d$  it follows that for every $k\geq0$ the sequence $(F_k^j)$ converges uniformly on $K_k$ to a holomorphic function $F_k$ as $j\rightarrow \infty$. A priori functions $F_k$ may not be defined outside $K_k$, but since $F_k^j(z)=E_{\lambda^j_{k+1}}(F^j_{k+1}(z))$ we will see that $F_k$ are in fact entire functions satisfying  $F_k(z)=E_{\lambda_{k+1}}({F_{k+1}(z)})$ for all $k\geq 0$, hence $F_0\in \mathcal{E}$.

Let $V_1\subset V_2$ be any compact sets in $\C$ and let $k\geq 0$.  It is sufficient to prove that if  $(F_k^j)$ converges uniformly to $F_k$ on $V_1$ and if $(F_{k+1}^j)$ converges uniformly to $F_{k+1}$ on $V_2$, then also  $(F_k^j)$ converges uniformly  to $F_k$ on $V_2$. Since  $F_k^j(z)=E_{\lambda^j_{k+1}}(F^j_{k+1}(z))$ for all $z\in\C$ we have 
$$|F_k^j(z)-F_k^{j+1}(z)|=\left|e^{F^{j+1}_{k+1}(z)}\right|\cdot\left|\lambda_{k+1}^je^{F_{k+1}^j(z)-F_{k+1}^{j+1}(z)}-\lambda_{k+1}^{j+1}\right|.$$
Since $(F_k^j)$ and $(F_{k+1}^j)$ both converge uniformly on $V_1$ this implies that the sequence $(\lambda_{k+1}^j)$ converges to some $\lambda_{k+1}\in\C$  as $j\rightarrow \infty$. Then since $(F_{k+1}^j)$ also converge uniformly on $V_2$ the above equality implies that also $(F_k^j)$ converges uniformly on $V_2$, hence 
$F_k(z)=E_{\lambda_{k+1}}(F_{k+1}(z))$ on $V_2$.

\end{proof}

Clearly Lemma \ref{unique} does not apply to constant functions as they can be represented by many different sequences $(\lambda_n)$ and $(F_n\equiv 1)$, hence the metric $d$ is not well defined for constant functions.  Next example shows that  our metric $d$ can not be extended over the entire class $\mathcal{E}$.  
\medskip

{\bf Example 1:} Let $(\lambda_n)$ and $(\mu_n)$ be two sequences of non-zero complex numbers where $\lambda_1\neq\mu_1$. Let $F_0,G_0\in \mathcal{E}$ be their associated functions given by Theorem \ref{main2}. It follows from the construction of these functions that  $F_0(0)=G_0(0)=1$. By the definition $F_0(z)=\lambda_1e^{F_1(z)}$ and $G_0(z)=\mu_1e^{G_1(z)}$,
hence $F_1(0)\neq G_1(0)$.  If we define $f_j=F_0(j^{-1}z)$   and $g_j=G_0(j^{-1}z)$ then clearly 
$f_j,g_j\in\mathcal{E}$ and $f_j,g_j\rightarrow 1$ as $j\rightarrow\infty$ but   $\lim_{j\rightarrow \infty} d\left(f_j,g_j\right)\neq 0.$  
\vspace{0.5cm}

Let $\mu\in\C^*$ and $\lambda=\mu e^{-\mu}$. We define the map $\Phi_\mu:\mathcal{E}\backslash \C\rightarrow \mathcal{E}\backslash \C$ as
$$\Phi_\mu(F(z)):=E_{\lambda}(F(\mu^{-1}z)).$$
Observe that for $|\mu|>1$  the fixed points of $\Phi_\mu$ are precisely the Poincar\'e functions \eqref{poincare1}.

\begin{proposition} Let $\mu\in\C$ satisfy $|\mu|>1$ and let $F\in \mathcal{E}\backslash \C$ satisfy $F(0)=\mu$ and $F'(0)\neq0$. The sequence of iterates $(\Phi_\mu^n(F))$ converges uniformly on compacts to the Poincar\'e function $f$ of  $E_{\lambda}$.
\end{proposition}

\begin{proof} Let $\mu$ and $F\in \mathcal{E}\backslash \C$ be a as in the proposition. Note that such function  exists by Theorem \ref{main2}. By $F^{-1}$ we denote a local inverse of $F$ defined on some neighbourhood of $\mu$  which satisfies $F^{-1}(\mu)=0$.  Let $\lambda=\mu e^{-\mu}$  and define
$$\phi(z)=(F^{-1}\circ E_\lambda \circ F)(z).$$
Note that $0$ is a repelling fixed point of $\phi$ with a multiplier $\mu$, hence there exists $r>0$ and a 
K\oe nigs function $\psi(z)=z+O(z^2)$ defined on closed disk $\Delta_r$ satisfying $\psi\circ\phi(z)=\mu\cdot\psi(z)$ for all $z\in \Delta_r$ \cite[Theorem 8.2]{milnor}.

Next we define 
$g_n(z)=\Phi_\mu^n(F(z))$ and observe that $g_n(z)=g_{n-1}\circ \tau_{n-1}\circ\phi\circ\tau_{n}^{-1}(z)$ 
where $\tau_n(z)=\mu^n(z)$. It follows that 
$$g_n(z)=F\circ\phi^n(\mu^{-n}z)=F\circ\psi^{-1}(\mu^n\cdot \psi( \mu^{-n} z)),$$
hence $g_n\rightarrow f:= F\circ\psi^{-1}$ uniformly on $\Delta_r$ as $n\rightarrow\infty$. Quick computation shows that $\Phi_\mu(f)=f$ on $\Delta_r$. We can extend the function $f$ holomorphically to the entire complex plane by the following simple trick. Let $w\in\C$ and let $n$ be so large that $\mu^{-n}w\in \Delta_r$, then we define $f(w)=\Phi_\mu^n(f(w))$. Since  $\Phi_\mu^k(g_n)=g_{n+k}$ one can easily deduce that the sequence $(g_n)$ also converges uniformly on compacts  to an entire function $f$ satisfying $\Phi_\mu(f)=f$, hence it is a Poincar\'e function.
\end{proof}

\begin{remark} If we assume that $F$ in the above proposition is a linear function, then this is a classical construction of a Poincar\'e function. We have seen that adding non-linear terms does not have any effect on the limiting function. The aim of introducing this proposition is to give an example how metric $d$ defined in \eqref{metric} and Lemma \ref{complete} could be used in certain constructions. We will illustrate this in the following paragraph.  
\end{remark}

{\it Alternative proof of Proposition 6.} Let $g_n$, $\tau_n$ and $\phi$ be defined as above. Define
$u_n=\tau_n\circ\phi\circ\tau^{-1}_{n+1}$ and observe that if $u_n$ is a linear function then $F$ is already a Pioncar\'e function. Assume that this is not the case, then let $k\geq 2$ be the smallest integer for which $u_n^{(k)}(0)\neq0$. Then there exists $a\neq0$ and $C,r>0$ such that 
$$|u_n(z)-z-a\mu^{-(k-1)n-k}z^k|\leq C|z^{k+1}||\mu^{-kn}|$$
 for all $z\in\Delta_r$ and all $n\geq0$. Next observe that $g_{n+1}=g_n\circ u_n$ and $g_{n-1}=g_n\circ u^{-1}_{n-1}$, hence we can write
$$\frac{|g_{n+1}(z)-g_{n}(z)|}{|g_{n}( z)- g_{n-1}(z)|}=\frac{\frac{|g_{n}(u_n(z))-g_{n}(z)|}{|u_{n}(z)- z|}}{\frac{|g_{n}(z)-g_{n}(u_{n-1}^{-1}(z))|}{| z- u_{n-1}^{-1}(z)|}}\frac{|u_{n}(z)- z|}{| z- u_{n-1}^{-1}(z)|}.$$
Let $\delta$ be so small that $\frac{1+\delta}{\mu}<1$ and recall that by the initial assumption we have $g_n'(0)=F'(0)\neq0$ for all $n\geq 0$. Therefore for every sufficiently small $r>0$
we have  
$$
\left|\frac{g_{n+1}(z)-g_{n}(z)}{g_{n}( z)- g_{n-1}(z)}\right|<\frac{1+\delta}{\mu}$$
for all $z\in \Delta_r$ and all $n\geq 0$. Furthermore there exists $0<\rho<1$ such that 
$$\frac{\|g_{n+1}-g_{n}\|_{\Delta_r}}{1+\|g_{n+1}-g_{n}\|_{\Delta_r}}\leq\rho\cdot \frac{\|g_{n}-g_{n-1}\|_{\Delta_r}}{1+\|g_{n}-g_{n-1}\|_{\Delta_r}} $$
for all $n\geq 0$.  Finally let $\epsilon>0$ be so small that $\theta:=\rho+\epsilon<1$, and let $K_n=\Delta_{r\cdot\mu^n}$ be a compact exhaustion  of $\C$.  Using this exhaustion $(K_n)$ end $\epsilon$ we define the metric $d$ as in \eqref{metric} and it is easy to verify that 
\begin{align*}d(g_{n+1},g_{n})&=\frac{\|g_{n+1}-g_{n}\|_{\Delta_r}}{1+\|g_{n+1}-g_{n}\|_{\Delta_r}}+\epsilon \cdot d(g_n,g_{n-1})\\
&\leq \rho\cdot \frac{\|g_{n}-g_{n-1}\|_{\Delta_r}}{1+\|g_{n}-g_{n-1}\|_{\Delta_r}}+\epsilon \cdot d(g_n,g_{n-1})\\
& \leq\theta \cdot d(g_n,g_{n-1}),
\end{align*}

 for all $n\geq0$. Since $0<\theta<1$ it follows that $(g_n)$ is a Cauchy sequence and hence by Lemma \ref{complete} converges uniformly on compacts to an entire function $f \in \mathcal{E}$ which satisfies $\Phi_\mu(f)=f$. Since $g_n'(0)=F'(0)\neq0$ for all $n\geq 0$ the limiting function $f$ is non-constant, hence it is a Poincar\'e function. $\Box$

\vspace{0.5cm}

We finish this section with the following example, showing that there are many functions in $\mathcal{E}$ whose Fatou set is non-empty. 

\medskip 

{\bf Example 2:} Let us show that for every $\lambda$ non-zero complex number there exist a function $f\in\mathcal{E}$ such that $f(1)=1$ and $f'(1)=\lambda$. Take any non-constant function $F_0=e^{F_1}\in\mathcal{E}$  and define $g_t(z)=\frac{F_0(tz)}{F_0(t)}$. Observe that $g_t(1)=1$ and that $g_t'(1)=\frac{F'_0(t)\cdot t}{F_0(t)}=F_1'(t)\cdot t$. Let $\mu_1$ and $\mu_2$ be solutions of the equation $z^2-\lambda=0$. Since $g_t'(1)$ is a non-constant entire function with respect to $t$ we know that it can omit at most one value, hence we may assume that there is a $t_0$ such that $g'_{t_0}(1)=\mu_1$. Finally we define  $f= g_{t_0}\circ g_{t_0}$ and it should be clear that this function satisfies desired properties. Note that  if $F_0\in\mathcal{B}$ then also $f\in\mathcal{B}$.

\section{Proof of Theorem \ref{main1}}

Let $V$ be a closed set containing $0$ and at least one more point. Let ${a_n}$ be an infinite sequence of points from $V$ which forms a dense subset of $V$, where $a_0=0$ and $a_n\neq0$ for $n\neq0$ (we allow points to repeat). For fixed $\lambda_1\ldots\lambda_{k-1}$ the function $E_{(0,k)}$ is a non-vanishing holomorphic entire function with respect to parameter $\lambda_n$, hence there exists a complex number $\lambda_n$ for which $E_{(0,k)}(0)=a_k$. By inductive procedure we can conclude that there exists a sequence of non-zero complex numbers $(\lambda_n)$  such that  $\{E_{(0,k)}(0)\mid k\geq 0\}$ is a dense subset of $V$. By  Theorem \ref{main2} there exists a sequence of locally univalent entire functions $(F_n)$ satisfying $F_n(z)=E_{\lambda_{n+1}}(F_{n+1}(z))$ for all $n\geq 0$, hence it follows that 
\begin{equation}\label{exp}
 F_{0}(z)=E_{0,n}(F_n(z))
\end{equation}
for all $n\geq 0$. The function $F_{0}$ is locally univalent therefore its set of critical values is empty and hence $S(F_{0})=\overline{AV(F_{0})}$. From  (\ref{eq}) and \eqref{exp} it follows that
$$\{E_{(0,k)}(0)\mid k\geq 0\}\subseteq AV(F_{0})$$
hence $V\subseteq S(F_0)$.

Recall $F_0$ is a uniform limit of a sequence of entire functions 
\begin{equation}\label{f}
f_{0,n}(z)=E_{(0,n)}\left(b_{n}+ \frac{z}{b_0\cdots b_{n-1}}\right)
\end{equation}
where $b_0:=1$ and $b_n:=2\pi m_n i+ \log b_{n-1}-\log{\lambda_{n}}$ and where the sequence of integers $(m_n)$ increases sufficiently fast. 

In what follows we will prove that for any sufficiently fast increasing sequence $(m_n)$  the set $\{E_{(0,k)}(0)\mid k\geq0\}$ is  dense in $AV(F_{0})$, and therefore $S(F_0)=V$.

\medskip

Before proceeding with the proof we need to introduce some additional notation. Since $E_\lambda:\C\rightarrow \C^*$ is a covering map, every point $z\in\C^*$ has a neighbourhood $U$ on which inverse branches of $E_\lambda$ are well defined univalent functions that can be expressed as
 $L_\lambda^k(z):=\log z- \log\lambda +2k\pi i$ (we always assume that $\log 1=0$). Let $I_n= (k_1,\ldots ,k_n)\in\mathbb{Z}^n$ and define 
$$
L_{I_n}(z):= L_{\lambda_n}^{k_n}\circ\ldots\circ L_{\lambda_1}^{k_1}(z).$$

Observe that the function $f_{0,n}$ is a covering map over $\mathbb{C}\backslash\{E_{0,k}(0)\mid 0\leq k<n \}$ hence  every $z\in \mathbb{C}\backslash\{E_{0,k}(0)\mid 0\leq k<n \}$ has a neighbourhood $U$ on which the inverse branches of $f_{0,n}$ are well defined and can be expressed as
$$g_{I_n}(z):=\left(L_{I_n}(z)-b_n\right)b_0\cdots b_{n-1},$$
where $I_n\in\mathbb{Z}^n$ and  $b_n$ are the same as in the definition of $f_{0,n}$.

\medskip

Let us define $\Omega:=\mathbb{C}\backslash\overline{\{E_{(0,k)}(0)\mid  k\geq 0\}}$ and recall that for every $n>0$ the function $f_{0,n}$ is a covering map over $\Omega$. Observe that every point $\zeta\in \Omega$ has a neighbourhood $U_{\zeta}\subset \Omega$ on which all inverse branches of $f_{0,n}$ are well defined univalent functions for all $n>0$.  Clearly this does not hold for points of the form $\zeta=E_{(0,\ell)}(0)$ since for all  $I_n=(0,\ldots ,0,m_{\ell+1},\ldots, m_n )$ with $n>\ell$ the inverse branch $g_{I_n}$ is not defined in $\zeta$. Observe that for  $\zeta\in\Omega$ the set $U_\zeta$ does not depend on the choice of the sequence $(m_n)$ in the definition of functions $f_{0,n}$.
\medskip

Since the sequence of locally univalent functions $f_{0,n}$ converges uniformly to a non-constant locally univalent function $F_0$, it follows that for a proper choice of a sequence $(I_n)_n$ with $I_n\in\mathbb{Z}^n$, the inverse branches $g_{I_n}$ converge locally uniformly in $\C\backslash S(F_0)$ to the inverse branch $g$ of $F_0$ . The following lemma tells us for which sequences $(I_n)$ this actually happens.

\begin{lemma}\label{lemmaseq}Let $(I_n)_n$ be a sequence with $I_n\in\mathbb{Z}^n$ and let $(m_n)$ be the sequence of integers from the definition of functions \eqref{f}. Let $\zeta\in \Omega$ and let  $U_\zeta$ be as above. If the sequence $(g_{I_n})$ converges uniformly on $U_\zeta$ to a holomorphic function $g$ then there exists $j\geq0$ and integers $k_1,\ldots, k_j\in\mathbb{Z}$ such that 
\begin{equation} \label{sequence}
I_n=(k_1\ldots,k_j,m_{j+1},\ldots,m_n)
\end{equation}
for all $n>j$. 
\end{lemma}
\begin{proof}
{\bf Step 1:} \emph{If  $g_{I_n}$ converges to $g$ then $I_n=(k_1(n),\ldots,k_{n-1}(n),m_n)$ for all sufficiently large $n$.} 

Since  $g_{I_n}\rightarrow g$ uniformly on $U_\zeta$ and $f_{0,n}\rightarrow F_0$ uniformly on compacts it follows that $id_{U_\zeta}=f_{0,n}\circ g_{I_n}\rightarrow F_0\circ g $, hence $g$ is univalent function on $U_\zeta$ and $F_0\circ g =id_{U_\zeta}$. 
Given $I_n= (k_1,\ldots ,k_n)\in\mathbb{Z}^n$ we define $I_n':= (k_1,\ldots ,k_{n-1})$ and observe that
\begin{equation}\label{formula}
g_{I_n}(z)=(L_{I_n}(z)-b_n)b_0\cdots b_{n-1}=\left(\log\left(\frac{L_{I_n'}(z)}{b_{n-1}}\right)+2\pi i(k_n-m_n)\right)b_0\cdots b_{n-1}.
\end{equation}
The sequence $(g_{I_n})$ converges uniformly on $U_\zeta$ therefore the sequence
$$\left|Im\left(\log\left(\frac{L_{I_n'}(z)}{b_{n-1}}\right)+2\pi i(k_n-m_n)\right)\right|\cdot|b_0\cdots b_{n-1}|$$
must stay bounded on $U_\zeta$. Since  $|b_0\cdots b_{n-1}|\rightarrow \infty$ as $n\rightarrow\infty$ this can only happen if there exists  $n_0\geq 0$ such that $k_n=m_n$ for all $n> n_0$. 

\medskip
{\bf Step 2}: \emph{If  $g_{I_n}\rightarrow g$ uniformly on $U_\zeta$ then also $g_{I_n'}\rightarrow g$ uniformly on $U_\zeta$ where $I_n'=(k_1(n),\ldots,k_{n-1}(n))$.}

First observe that
\begin{equation}\label{gn}
g_{I_{n}}(z)=\left(L_{\lambda_{n}}^{k_{n}}\left(\frac{g_{I_n'}(z)}{b_0\cdots b_{n-2}}+b_{n-1}\right)-b_{n}\right)b_0\cdots b_{n-1}.
\end{equation}
This equation implies that
$$g_{I_{n}'}(z)=\left(\exp\left(\frac{g_{I_n}(z)}{b_0\cdots b_{n-1}}\right)-1\right)b_0\cdots b_{n-1}=g_{I_n}(z)+O\left(\frac{g_{I_n}^2(z)}{b_0\cdots b_{n-1}}\right).$$
Since the sequence $g_{I_n}\rightarrow g$ uniformly on $U_\zeta$   and $|b_0\cdots b_{n-1}|\rightarrow \infty$ it follows that $O\left(\frac{g_{I_n}^2(z)}{b_0\cdots b_{n-1}}\right) \rightarrow 0 $,  hence also the sequence  $g_{I_n'}\rightarrow g$ uniformly on $U_\zeta$. 
\medskip

{\bf Step 3:} \emph{Let $(I_n)$ and $(J_n)$ be two sequences for which both $g_{I_n}$ and $g_{J_n}$ converge uniformly on $U_\zeta$ to the univalent function $g$, then $I_n=J_n$ for all sufficiently large $n$.}

For every $n$ functions $g_{I_n}$ and $g_{J_n}$ are inverse branches of $f_{0,n}$ on $U_\zeta$, hence we know that that 
\begin{equation}\label{g}
g_{I_n}(U_\zeta)\cap g_{J_n}(U_\zeta)\neq\emptyset\iff I_n=J_n.
\end{equation}
Since both sequences converge to the same univalent function $g$, it follows that  for all  sufficiently large $n$ sets $g_{I_n}(U_\zeta)$ and $g_{J_n}(U_\zeta)$ are small perturbation of the set $g(U_\zeta)$. Therefore there exists $n_0\geq0$ such that  
$g_{I_n}(U_\zeta)\cap g_{J_n}(U_\zeta)\neq\emptyset$ for all $n>n_0$. 
\medskip

Let us summarize what we have just proven. If  $g_{I_n}$ converges to $g$ then 
$I_n=(k_1(n),\ldots,k_{n-1}(n),m_n)$ for all sufficiently large $n$. Next we have seen that also $g_{I_n'}$ converges to $g$ where $I_n'=(k_1(n),\ldots,k_{n-1}(n))$. Finally we have proven that since $(g_{I_n})$ and $(g_{I_n'})$ converge to the same limit map $g$  this means that $I_n'=I_{n-1}$ for all sufficiently large $n$, hence we obtain \eqref{sequence}.
\end{proof}
\medskip

It remains to prove that if the sequence of integers $(m_n)$ increases sufficiently fast (note that we can always achieve this, see the proof of Theorem \ref{main2}) then for every point $\zeta\in\Omega$ and for every  sequence $(I_n)$ of the form \eqref{sequence} the sequence $(g_{I_n})$ converges uniformly on $U_\zeta$ to a univalent holomorphic function.
\medskip

First we define
$$\mathcal{I}_n:=\{(\ell_1,\ldots,\ell_{n-1},m_n)\in\mathbb{Z}^n\mid |\ell_j|\leq m_{n-1} \text{ for all } 0 < j<n \}$$
and observe that since $m_n\rightarrow\infty$ as $n\rightarrow\infty$  it follows that for every sequence $(I_n)$ of the form \eqref{sequence} there exists $n_0\geq 0$  so that $I_n\in \mathcal{I}_n$ for all $n>n_0$. 

Let $K_1\subset K_2\subset\ldots\subset\bigcup_{n\geq 1}K_n=\Omega$ be an exhaustion by compacts and observe that for any $n>0$ there exists finitely points  
$\zeta^n_1,\ldots,\zeta_{j_n}^n \in K_n$ such that $K_n\subset\bigcup_{k=1}^{j_n}U_{\zeta^n_k}$. Since for every $n>0$ the set $\mathcal{I}_n$ is finite and since $U_\zeta$ does not depend on $I\in \mathcal{I}_n$ it follows that there exists a constant $C_n>0$ such that $\|g_{I}\|_{U_{\zeta^n_k}}<C_n$ for all $I\in\mathcal{I}_n$ and for all $1\leq k\leq j_n$, hence we can write 
$$\|g_{I}\|_{K_n}<C_n.$$  
 Let $I_n=(\ell_1,\ldots,\ell_{n-1},m_n)\in \mathcal{I}_n$ and observe as in \eqref{formula} that on each $U_{\zeta^n_k}$ we have
$$g_{I_n}(z)=\log\left(\frac{L_{I_n'}(z)}{b_{n-1}}\right)b_0\cdots b_{n-1}$$
where $I_n'=(\ell_1,\ldots,\ell_{n-1})$. Since there is no $b_n$ in the above expression it is clear that $C_n$ does not depend on $m_n$, hence for every $\epsilon>0$ there exists $N_{\epsilon}>0$ so that  
\begin{equation}\label{eq1}\frac{\|g_{I_n}\|_{K_n}}{|b_0\cdots b_n|}< \frac{C_n}{|b_n|}<\epsilon
\end{equation}
 for all  $m_n>N_\epsilon$.

\medskip

Finally let  $I_n=(k_1,\ldots,k_n)\in\mathbb{Z}^n$ be the sequence of the form \eqref{sequence}, hence there exists $n_0\geq0$ such that $k_n=m_n$ and  $I_{n+1}'=I_n\in \mathcal{I}_n$ for all $n>n_0$. Assuming that $m_n$ are sufficiently large it follows from \eqref{eq1} and \eqref{gn}  that  for all $n>n_0$ we have 
\begin{equation}\label{gn2}
\|g_{I_{n+1}}-g_{I_n}\|_{K_n}=\left\|b_0\cdots b_n \log\left(1+\frac{g_{I_n}(z)}{b_0\cdots b_n}\right)-g_{I_n}(z)\right\|_{K_n}\leq  2^{-n}.
\end{equation}
Note that this computation is actually made on each set $U_{\zeta^n_k}$ for $1\leq k\leq j_n$ separately, hence the bound holds on entire $K_n$.

\medskip
This proves that for every $\zeta\in\Omega$ the inverse branches $g_{I_n}$ (where the sequence $(I_n)$ has to be of the form \eqref{sequence} by Lemma \ref{lemmaseq}) of $f_{0,n}$ converge uniformly on $U_\zeta$ to the inverse branch of $F_0$ on $U_\zeta$ (note that every inverse branch of $F_0$ on $\Omega$ can be obtained this way), hence 
$$S(F_0)\subseteq \overline{\{E_{(0,k)}(0)\mid  k\geq 0\}}=V.$$

Recall that in the first paragraph of this section we have already proven $V\subseteq S(F_0)$, hence we finally obtain $S(F_0)=V$.  This completes the proof of Theorem \ref{main1}.

\section{Concluding remarks}

The following two corollaries are immediate consequences of Theorem \ref{main1}.

\begin{corollary} Given any closed set $U$ containing $0$, the family $Exp=\{E_{\lambda}(z)\mid \lambda\in\C\}$ lies in the closure
 of $\mathbb{S}_U=\{f\in\mathcal{O}(\C)\mid S(f)=U\}$.
\end{corollary}
\begin{proof}
 Observe that $f\equiv 0\in \overline{\mathbb{S}_U}$ as we can take any $g\in\mathbb{S}_U$ and define $g_n(z)=\frac{1}{n}g(z)\in\mathbb{S}_U$ which clearly converges to $0$ as $n\rightarrow\infty$.  If $U=\{0\}$ then $Exp\backslash\{0\}\subset \mathbb{S}_U$. Let  $U$ be a closed set containing the origin and at least one more point. By  Theorem \ref{main1} there exists a sequence $(f_n)\in \mathcal{E}$ satisfying $S(f_n)=U$ and 
 $|f_n(z)-E_1(z)|\leq 2^{-n}$ on $\Delta_n:=\{z\in\C\mid |z|\leq n\}$ for all $n\geq 1$, hence for $\lambda\neq 0$ the sequence $(\lambda\cdot f_n)\in \mathbb{S}_U$ converges uniformly on compacts to the function $E_\lambda(z)$
\end{proof}
\begin{corollary} The closure of the class $\mathcal{E}$ is equal to $\mathcal{O}^*(\C)\cup\{0\}$. \footnote{By $\mathcal{O}^*(\C)$ we denote the space of all non-vanishing entire functions.}
\end{corollary}
\begin{proof}
Let $g\in\mathcal{O}^*(\C)\cup\{0\}$. If $g$ is constant then it is already contained in $\mathcal{E}$. Assume now that $g$ is not constant hence there is a entire function $h$ such that $g(z)=E_1(h(z))$. By  Theorem \ref{main1} there exists a sequence $(f_n)\in \mathcal{E}$ which converges uniformly on compacts to $E_1(z)$, hence the sequence $(f_n\circ h)\in \mathcal{E}$ converges uniformly on compacts to $g$.

 Finally assume that  $g$ is a non-constant entire function which lies in the closure of $\mathcal{E}$ and it vanishes at some point $z_0$. There exists a small closed disk $\Delta(z_0,r)$ centred at $z_0$ such that $g$ does not vanish on $\Delta(z_0,r)-\{z_0\}$. By our assumption there exists a sequence $(f_n)\in \mathcal{E}$ which converges uniformly on compacts to $g$. Let $0<\epsilon < \min_{\partial \Delta(z_0,r)}|g|$ and let $n_0$ be sufficiently large such that $|f_{n_0}-g|<\epsilon$ on $\Delta(z_0,r)$. By Rouch\'e's theorem  functions $f_{n_0}$ and $g$ have the same number of zeros in $\Delta(z_0,r)$, which is a contradiction.
\end{proof}

\medskip
We end this paper with the following example which shows that functions in class $\mathcal{E}$ can have an empty Fatou set.

\medskip
 
{\bf Example 3:} Let $U=\{0,1\}$ and $0<\epsilon<1$. By Theorem \ref{main1} there exists a locally univalent function $F_0\in\mathcal{E}$ satisfying $S(F_0)=U$ and $\|F_0(z)-E_1(z)\|_{\Delta_{4\pi}}\leq\epsilon$. It follows from the construction of such function that we may assume $F_0(0)=1$ , see the proof of Theorem \ref{main2}. If $\epsilon$ was chosen sufficiently small, then using Rouch\'e's theorem, we can prove that  there exists a  $\lambda\in\C$ satisfying $|\lambda -2\pi i|\leq 1/2$ and $F_0(\lambda)=1$. By defining $f(z):=F_0(\lambda z)$ we obtain  a locally univalent function in $\mathcal{E}$ that satisfies $S(f)=\{0,1\}$ and $f(0)=f(1)=1$,  hence the postsingular set of $f$ is finite. Moreover using Cauchy estimates we obtain $|f'(1)|>1$, hence the Fatou set of $f$ is empty. 
 \medskip

 \section*{Acknowledgments} The author would like to thank Han Peters for fruitful discussions in the early stages of this project. This project was supported by the research program P1-0291 from ARRS, Republic of Slovenia.

\end{document}